\documentclass[12pt]{article} \textheight 8.8 true in \textwidth
6.2 true in

\usepackage[T2A]{fontenc}
\usepackage[cp1251]{inputenc}
\usepackage[english]{babel}
\usepackage{amsfonts}
\usepackage{amssymb,amsmath,amstext}
\usepackage{amscd, amsthm, latexsym}
\usepackage{color}
\usepackage[final]{graphicx}
\usepackage{epsfig}
\usepackage{euscript}
\usepackage{srcltx,epsf}
\usepackage{wrapfig}
\usepackage{mathrsfs}
\usepackage{ifthen}
\usepackage{cleveref}
\usepackage{multirow}
\usepackage{array}
\usepackage{slashbox}

\def\hide#1{}
\def\old#1{}

\def\oop#1{}
\def\gap#1{}

\theoremstyle{plain}
\newtheorem{theorem}{Theorem}%[section]
\newtheorem{proposition}{Proposition}%[section]
\newtheorem*{theorem*}{Theorem}
\newtheorem{lemma}{Lemma}%[section]
\newtheorem{remark}{Remark}
\newtheorem{corollary}{Corollary}
\newtheorem{definition}{Definition}

\def \ZZ  {\Bbb Z}

\begin{document}

%\newcounter{figcounter}
%\setcounter{figcounter}{0} \addtocounter{figcounter}{1}

\title{On the coverings of closed orientable Euclidean manifolds $\mathcal{G}_{2}$ and $\mathcal{G}_{4}$.}
\author{
G.~Chelnokov\thanks{This work was supported by the Russian Foundation for Basic Research (grant 18-01-00036).}\\
{\small\em National Research University Higher School of Economics} \\
{\small\tt grishabenruven@yandex.ru }
\\ [2ex]
%M.~Deryagina\thanks{This work was supported by the Russian Foundation for Basic Research (grant 13-01-00513).}\\
%%{\small\em Sobolev Institute of Mathematics, Russia}\\
%%{\small\em Moscow State University of Technologies and} \\
%%{\small\em Management named after K.G. Razumovskiy, Russia}\\
%%{\small\tt madinaz@rambler.ru }
%\\ [2ex]
A.~Mednykh\thanks{This work was supported by the Russian Foundation
for Basic Research (grant 19-41-02005).}\\
{\small\em Sobolev Institute of Mathematics,} \\
{ \small\em Novosibirsk, Russia}\\
{\small\em Novosibirsk State University,}\\
{\small\em Novosibirsk, Russia}\\
{\small\tt mednykh@math.nsc.ru} }
\date{}
\maketitle \maketitle

\begin{abstract}
There are only 10 Euclidean forms, that is  flat closed three
dimensional manifolds: six are orientable and four are
non-orientable. The aim of this paper is to describe all types of
$n$-fold coverings over orientable Euclidean manifolds
$\mathcal{G}_{2}$ and $\mathcal{G}_{4}$, and calculate the numbers
of non-equivalent coverings of each type. We classify subgroups in
the fundamental groups $\pi_1(\mathcal{G}_{2})$ and
$\pi_1(\mathcal{G}_{4})$ up to isomorphism and calculate the numbers
of conjugated classes of each type of subgroups for index $n$. The
manifolds $\mathcal{G}_{2}$ and $\mathcal{G}_{4}$ are uniquely
determined among the others orientable forms by their homology
groups $H_1(\mathcal{G}_{2})=\ZZ_2\times \ZZ_2 \times \ZZ$ and
$H_1(\mathcal{G}_{4})=\ZZ_2 \times \ZZ$.


\end{abstract}

\section*{Introduction}

Consider a manifold $\mathcal{M}$. Two coverings $p_1: \mathcal{M}_1
\to \mathcal{M}  $ and $p_2: \mathcal{M}_2 \to \mathcal{M}  $ are
said to be equivalent if there exists a homeomorphism $h:
\mathcal{M}_1 \to \mathcal{M}_2$ such that $p_1 =p_2 \circ h.$
According to the general theory of covering spaces, any $n$-fold
covering is uniquely determined by a subgroup of index $n$ in the
group $\pi_1(\mathcal{M})$. The equivalence classes of $n$-fold
covering of $\mathcal{M}$ are in one-to-one correspondence with the
conjugacy
 classes of subgroups of index $n$ in the fundamental group
 $\pi_1(\mathcal{M}).$ (See,
for example, \cite{Hatcher}, p.~67). A similar statement formulated
in the language of orbifolds is valid for branched coverings.

  In such a way the following two natural problems arise. The first one is to calculate the
  number $s_{\Gamma}(n)$
 of subgroups of given finite index $n$ in $\Gamma=\pi_1(\mathcal{M})$. The second problem is to find the number $c_{\Gamma}(n)$ of conjugacy
 classes of subgroups of index $n$ in $\Gamma$.

If $\mathcal{M}$  is a compact surface with nonempty boundary of
Euler characteristic $\chi(\mathcal{M}) = 1 - r,$ where $r \ge 0 $,
then its fundamental group $ \pi_1(\mathcal{M}) = F_r$ is the free
group of rank $r.$ For this case, M.~Hall \cite{Ha49} calculated the
number $s_\Gamma(n)$ and V.~A.~Liskovets \cite{Li71} found the
number $c_\Gamma(n)$ by using his own method for calculating the
number of conjugacy classes of subgroups in free groups. An
alternative approach for counting conjugacy classes of subgroups in
$F_r$  was suggested in \cite{KwakLee96}. The numbers $s_\Gamma(n)$
and $c_\Gamma(n)$  for the fundamental group of a closed surface
(orientable or not) were calculated in (\cite{Med78}, \cite{Med79},
\cite{MP86}). In the paper \cite{Medn}, a general method for
calculating the number $c_\Gamma(n)$  of conjugacy classes of
subgroups in an arbitrary finitely generated group $\Gamma$  was
given.  Asymptotic formulas for $s_\Gamma(n)$ in many important
cases were obtained by T.~W.~M\"uller and his collaborators
(\cite{Mul},  \cite{MulShar}, \cite{MulPuch}).

In the three-dimensional case, for a large class of Seifert fibrations, the value of $s_\Gamma(n)$ was determined in \cite{LisMed00} and \cite{LisMed12}. In the previous paper by the authors \cite{CheDerMed}, the numbers $s_\Gamma(n)$ and $c_\Gamma(n)$  were determined for the fundamental groups of non-orientable Euclidian manifolds   $\mathcal{B}_1$ and $\mathcal{B}_2$ whose homologies are ${H}_1(\mathcal{B}_1)= \mathbb{Z}_2\times\mathbb{Z}^2$ and ${H}_1(\mathcal{B}_2)=  \mathbb{Z}^2.$

 The aim of the present paper is to investigate
$n$-fold coverings over orientable Euclidean three dimensional
manifolds $\mathcal{G}_{2}$ and $\mathcal{G}_{4}$, whose homologies
are $H_1(\mathcal{G}_{2})=\ZZ_2\times \ZZ_2 \times \ZZ$ and
$H_1(\mathcal{G}_{4})=\ZZ_2 \times \ZZ$.  We classify subgroups of
finite index in the fundamental groups of $\pi_1(\mathcal{G}_{2})$
and $\pi_1(\mathcal{G}_{4})$ up to isomorphism and calculate the
numbers of conjugated classes of each type of subgroups for index
$n$.

 We note that numerical methods to solve these and similar problems for the three-dimensional crystallogical  groups were
developed by the Bilbao group \cite{babaika}. The first homologies of such groups are determined in \cite{Ratc}.

.

\subsection*{Notations}
According to \cite{Wolf}, there are six orientable Euclidean
3-manifold $\mathcal{G}_{1}$, $\mathcal{G}_{2}$, $\mathcal{G}_{3}$,
$\mathcal{G}_{4}$, $\mathcal{G}_{5}$, $\mathcal{G}_{6}$, and four
non-orientable ones $\mathcal{B}_{1}$, $\mathcal{B}_{2}$,
$\mathcal{B}_{3}$, $\mathcal{B}_{4}$. One can find the
correspondence between Wolf and Conway-Rossetti notations of the
Euclidean 3-manifold and its fundamental group in Table~1 in
\cite{CheDerMed}.

We  use the following notations: $s_{H,G}(n)$ is the number of
subgroups of index $n$ in the group $G$, isomorphic to the group
$H$; $c_{H,G}(n)$ is the number conjugacy classes of subgroups of
index $n$ in the group $G$, isomorphic to the group $H$. Also we
will need the following number-theoretic functions. In all cases we
consider the function
$$\sigma_0(n) = \sum_{k \mid n}1,\quad\sigma_1(n) =  \sum_{k \mid n} k,\quad \sigma_2(n) =  \sum_{k \mid n} \sigma_1(k),\quad \omega(n) = \sum_{k \mid n} k\sigma_1(k),$$
%$$\sigma_1(n) =  \sum_{k \mid n} k,$$
%$$\sigma_2(n) =  \sum_{k \mid n} \sigma_1(k),$$
%$$\omega(n) = \sum_{k \mid n} k\sigma_1(k),$$
$$
\tau(n)=|\{(s,t)|s,t \in \ZZ, s>0, t\ge 0, s^2+t^2=n\}|.
$$
Also we suppose the above functions vanish if $n\not\in \Bbb{N} $.
% In all cases we consider the function equals if $n\not\in \Bbb{N} $.
For the first three, see for exapmle, (\cite{Apostol} Ch. 2.13., p.
38). For the functions $\omega(n)$ and $\tau(n)$ see the sequences
A001001 and A002654 in \cite{Ency}.

\section{The brief overview of achieved results}
Since the problem of enumeration of $n$-fold coverings reduces to
the problem of enumeration of conjugacy classes of some subgroups,
it is natural to expect that the enumeration of subgroups without
respect of conjugacy would be helpful. The manifold
$\mathcal{G}_{1}$ have the Abelian fundamental group $\ZZ^3$. Thus
the number of subgroups of a given finite index $n$ coincides with
the number of conjugacy classes and well known:
$$
s_{\ZZ^3,\ZZ^3}(n)=c_{\ZZ^3,\ZZ^3}(n)=\omega(n).
$$

In this paper we enumerate subgroups of a given finite index $n$ and
conjugacy classes of such subgroups with respect of their
isomorphism class in groups $\pi_{1}(\mathcal{G}_{2})$ and
$\pi_{1}(\mathcal{G}_{4})$. Similar results for manifolds
$\mathcal{B}_{1}$ and $\mathcal{B}_{2}$ are achieved in
\cite{CheDerMed}. Analogous results for other five Euclidean
3-manifolds are coming soon.

The first theorem provides the complete solution of the problem of
enumeration of subgroups of a given index in
$\pi_{1}(\mathcal{G}_{2})$.

\begin{theorem}\label{th-1-dicosm}
Every subgroup $\Delta$ of finite index $n$ in
$\pi_{1}(\mathcal{G}_{2})$ is isomorphic to either
$\pi_{1}(\mathcal{G}_{2})$ or $\ZZ^3$. The respective numbers of
subgroups are
$$
s_{\pi_1(\mathcal{G}_{2}), \pi_1(\mathcal{G}_{2})}(n) =
\omega(n)-\omega(\frac{n}{2}),\leqno (i)
$$
$$
s_{\ZZ^3, \pi_1(\mathcal{G}_{2})}(n) =  \omega(\frac{n}{2});\leqno
(ii)
$$
%where $\omega(n)=\sum_{k\mid n}k\sigma_1(k)$.
\end{theorem}
The next theorem provides the number of conjugacy classes of
subgroups of index $n$ in $\pi_{1}(\mathcal{G}_{2})$ for each
isomorphism type. That is the number of non-equivalent $n$-fold
covering $\mathcal{G}_{2}$, which have a prescribe fundamental
group.
%thus the
%number of $n$-fold coverings of $\mathcal{G}_{2}$ with respect of
%the isomorphism type of a cover.
\begin{theorem}\label{th-2-dicosm}
Let $\mathcal{N} \to \mathcal{G}_{2}$ be an $n$-fold covering over
$\mathcal{G}_{2}$. If $n$ is odd then $\mathcal{N}$ is homeomorphic
to $\mathcal{G}_{2}$. If $n$ is even then $\mathcal{N}$ is
homeomorphic to $\mathcal{G}_{2}$ or $\mathcal{G}_{1}$. The
corresponding numbers of nonequivalent coverings are given by the
following formulas:
$$
c_{\pi_{1}(\mathcal{G}_{2}),\pi_{1}(\mathcal{G}_{2})}(n)
=\sigma_2(n)+2\sigma_2(\frac{n}{2})-3\sigma_2(\frac{n}{4}). \eqno
(i)
$$
$$
c_{\ZZ^3,\pi_{1}(\mathcal{G}_{2})}(n) =
\frac{1}{2}\Big(\omega(\frac{n}{2})+\sigma_2(\frac{n}{2})+3\sigma_2(\frac{n}{4})\Big).
\eqno (ii)
$$
\end{theorem}

  \Cref{th-1-tetracosm} and \Cref{th-2-tetracosm} count  the numbers of
subgroups and the numbers of conjugacy classes of subgroups of index
$n$ in $\pi_{1}(\mathcal{G}_{4})$. That is this theorems are
analogues of \Cref{th-1-dicosm} and \Cref{th-2-dicosm} respectively
for the manifold $\mathcal{G}_4$. We will need one more
combinatorial function.

%\smallskip
%{\bf Notation.} Denote $\tau(m)=|\{(s,t)|s,t \in \ZZ, s>0, t\ge 0,
%s^2+t^2=m\}|$.
%\smallskip

\begin{theorem}\label{th-1-tetracosm}
Every subgroup $\Delta$ of finite index $n$ in
$\pi_{1}(\mathcal{G}_{4})$ is isomorphic to either
$\pi_{1}(\mathcal{G}_{4})$, or $\pi_{1}(\mathcal{G}_{2})$, or
$\ZZ^3$. The respective numbers of subgroups are
$$
s_{\pi_1(\mathcal{G}_{4}), \pi_1(\mathcal{G}_{4})}(n) =\sum_{a \mid
n}a\tau(a)-\sum_{a \mid \frac{n}{2}}a\tau(a).\leqno (i)
$$
$$
s_{\pi_1(\mathcal{G}_{2}), \pi_1(\mathcal{G}_{4})}(n) =
\omega(\frac{n}{2})-\omega(\frac{n}{4}),\leqno (ii)
$$
$$
s_{\ZZ^3, \pi_1(\mathcal{G}_{4})}(n) = \omega(\frac{n}{4}) .\leqno
(iii)
$$
\end{theorem}

\begin{theorem}\label{th-2-tetracosm}
Let $\mathcal{N} \to \mathcal{G}_{4}$ be an $n$-fold covering over
$\mathcal{G}_{4}$. If $n$ is odd then $\mathcal{N}$ is homeomorphic
to $\mathcal{G}_{4}$. If $n$ is even but not divisible by 4 then
$\mathcal{N}$ is homeomorphic to $\mathcal{G}_{4}$ or
$\mathcal{G}_{2}$. Finely, if $n$ is divisible by 4 then
$\mathcal{N}$ is homeomorphic to one of $\mathcal{G}_{4}$,
$\mathcal{G}_{2}$ and $\mathcal{G}_{1}$. The corresponding numbers
of nonequivalent coverings are given by the following formulas:
$$
c_{\pi_1(\mathcal{G}_4),\pi_1(\mathcal{G}_4)}(n)=\sum_{a\mid
n}\tau(a) - \sum_{a\mid \frac{n}{4}}\tau(a) \leqno (i)
$$
$$
c_{\pi_1(\mathcal{G}_2),\pi_1(\mathcal{G}_4)}(n)=\frac{1}{2}\Big(\sigma_2(\frac{n}{2})+2\sigma_2(\frac{n}{4})-3\sigma_2(\frac{n}{8})+\sum_{a\mid
\frac{n}{2}} \tau(a)-\sum_{a\mid \frac{n}{8}} \tau(a)\Big), \leqno
(ii)
$$
$$
c_{\ZZ^3,\pi_1(\mathcal{G}_4)}(n)=\frac{1}{4}\Big(\omega(\frac{n}{4})+\sigma_2(\frac{n}{4})+3\sigma_2(\frac{n}{8})+2\sum_{a\mid\frac{n}{4}}\tau(a)+2\sum_{a\mid\frac{n}{8}}\tau(a)\Big).
\leqno (iii)
$$
\end{theorem}

In Appendix 1 we tabulate the functions given by theorems 1--4. In
Appendix 2 we present the obtained results in terms of Dirichlet
generating functions.

\section{Preliminaries}
Further we use the following representations for the fundamental
groups $\pi(\mathcal{G}_2)$ and $\pi(\mathcal{G}_4)$, see
\cite{Wolf} or \cite{Conway}.

\begin{equation}\label{fund_G2}
 \pi_{1}(\mathcal{G}_{2})=\langle x, y, z: xyx^{-1}y^{-1}=1,
zxz^{-1}= x^{-1}, zyz^{-1}=y^{-1}
 \rangle .
 \end{equation}
\begin{equation}\label{fund_G4}
 \pi_{1}(\mathcal{G}_{4})=\langle \tilde{x}, \tilde{y}, \tilde{z}: \tilde{x}\tilde{y}\tilde{x}^{-1}\tilde{y}^{-1}=1,
\tilde{z}\tilde{x}\tilde{z}^{-1}= \tilde{y},
\tilde{z}\tilde{y}\tilde{z}^{-1}=\tilde{x}^{-1}
 \rangle .
 \end{equation}

We widely use the following statement.

\begin{proposition}\label{number of sublattices}
The sublattices of index $n$ in the $2$-dimensional
lattice $\ZZ^2$ are in one-to-one correspondence with the matrices $\begin{pmatrix}   b & c \\
 0 & a \end{pmatrix}$, where $ab=n$, $0 \le c < b$. Consequently,
the number of such sublattices is $\sigma_1(n)$.

The sublattices of index $n$ in the $3$-dimensional
lattice $\ZZ^3$ are in one-to-one correspondence with the integer matrices $\begin{pmatrix} c & e & f \\
0 & b & d \\ 0 & 0 & a\end{pmatrix}$, where $a,b,c >0, \,abc=n$, $0
\le d < b$ and $0 \le f,e < c$. Consequently, the number of such
sublattices is $\omega(n)$.
\end{proposition}
Despite this statement is well-known we will quote here its proof,
since we will need the technique for a number of more subtle
questions.

\begin{proof}
We prove the statement (ii), the proof of statement (i) is similar.
Consider the group $\ZZ^3=\{(x,y,z)| \; x,y,z \in \ZZ\}$, we use the
additive notation. Let $\Delta$ be a subgroup of index $n$ in
$\ZZ^3$.

Take some generators of the subgroup $\Delta$ in such a way that
they form an upper triangular matrix, that is $\Delta=\langle
\bar{u}, \bar{v},\bar{w} \rangle$ where $\bar{w}=(f,d,a)$,
$\bar{v}=(e,b,0)$, $\bar{u}=(c,0,0)$. This is a standard Gauss
procedure. Replacing $\bar{w}$ with $\bar{w}+i\bar{v}+j\bar{u}\;i,j
\in \ZZ$ we may assume $0\le d <b$ and $0\le f <c$. Similarly $0\le
e <c$. Since $\Delta$ is a subgroup of index $n$, holds $abc=n$.

Thus we got the map form the subgroups to matrices. Both the
injectivity and the surjectivity are obvious.
\end{proof}
%\begin{corollary}\label{number of halfes}
%{\bf Old Variant} Given an integer $n$, denote by $S(n)$ the sum of
%the number of cosets $\nu \in \ZZ^2/H$ with $2\nu=0$ (we use the
%additive notation) taken over all subgroups $H$ of index $n$ in
%$\ZZ^2$ . Then
%$$
%S(n)=\sigma_1(n)+3\sigma_1(\frac{n}{2}).
%$$
%\end{corollary}
\begin{corollary}\label{number of halfes}
 Given an integer $n$, by $S(n)$ denote the number
of pairs $(H,\nu)$, where $H$ is a subgrouup of index $n$ in $\ZZ^2$
and $\nu$ is a coset of $\ZZ^2/H$ with $2\nu=0$ (we use the additive
notation). Then
$$
S(n)=\sigma_1(n)+3\sigma_1(\frac{n}{2}).
$$
\end{corollary}
\begin{proof}
We have to enumerate pairs $(H,\nu)$ such that $H$ is a subgroup of
index $n$ in $\ZZ^2$ and $\nu$ is a coset in $\ZZ^2/H$ with
$2\nu=0$. Consider a pair of this type. Let $g$ be an arbitrary
representative of $\nu$. Put $G=\langle g,H\rangle$. Obviously, $G$
does not depends upon a choice of $g$.

Since $2\nu=0$, two cases are possible: either $G=H$ or $|G:H|=2$.
In the first case $\nu=0$, then $G$ uniquely identifies a pair
$(H,\nu)$. That is the number of pairs matching the first case is
$\sigma_1(n)$ by \Cref{number of sublattices}. In the second case
$|\ZZ^2:G|=\frac{n}{2}$. Then there are $\sigma_1(\frac{n}{2})$
choices of $G$. Also $G\cong\ZZ^2$ as a subgroup of finite index in
$\ZZ^2$. So, by \Cref{number of sublattices}, in each $G$ there are
$\sigma_1(2)=3$ subgroups $H$ with $|G:H|=2$. Note that a pair
$(G,H)$ uniquely identifies a pair $(H,\nu)$. That is the number of
pairs $(H,\nu)$ matching the second case is
$3\sigma_1(\frac{n}{2})$.
\end{proof}

Let us note the following fact.

\begin{lemma}\label{lemG4-1.2}\label{index_and_ker}
Let $H\leqslant G$ be an abelian group and its subgroup of finite
index. Let $\phi: G \to G$ be an endomorphism of $G$, such that
$\phi(H)\leqslant H$ and the index $|G:\phi(G)|$ is also finite.
%Since the action of $\phi$ on the cosets of $G/H$ is
%well-defined, we allow ourselves not to introduce a new designation.
Then the cardinality of kernel of $\phi: G/H \to G/H$ equals to the
index $|G : (H+\phi(G))|$.
\end{lemma}
\begin{proof}
Indeed,
$$
|\ker_{\phi}(G/H)|=\frac{|G/H|}{|\phi(G)/(\phi(G)\bigcap H)|}
$$

By the Second Isomorphism Theorem
$\phi(G)/(\phi(G)\cap H)\cong
(\phi(G)+H)/H$. So
$$
|\ker_{\phi}(G/H)|=\frac{|G/H|}{(\phi(G)+H)/H}=|G/(\phi(G)+H)|.
$$
\end{proof}

\begin{remark}\label{number of halfes-remark}
Combining \Cref{index_and_ker} and \Cref{number of halfes} we get
the following observation. Given a subgroup $H\leqslant \ZZ^2$, the
number of $\nu \in \ZZ^2/H$, such that $2\nu=0$, is equal to
$|\ZZ^2/\langle (2,0), (0,2), H\rangle|$. Indeed, taking $\ZZ^2$ as
$G$, $H$ as $H$ and $\phi: g\to 2g,\, g\in \ZZ^2$ as $\phi$ one gets
the desired equality. Since for each $H$ the numbers $|\{\nu| \nu\in
\ZZ^2/H,\,2\nu=0\}|$ and $|\ZZ^2/\langle (2,0), (0,2), H\rangle|$
coincide, their sums taken over all subgroups $H$ also coincide,
that is
$$
S(n)=\sum_{H\leqslant \ZZ^2,\,|\ZZ^2/H|=n}|\{\nu| \nu\in
\ZZ^2/H,\,2\nu=0\}|=\sum_{H\leqslant \ZZ^2,\,|\ZZ^2/H|=n}
|\ZZ^2/\langle (2,0), (0,2), H\rangle|.
$$
\end{remark}

\begin{definition}\label{ell}
Consider the group $\ZZ^2$. By $\ell$ denote the automorphism $\ell:
\ZZ^2 \to \ZZ^2$ given by $(x,y) \to (-y,x)$.
\end{definition}

\begin{lemma}\label{lemG4-1.2}\label{lemG4-1.3}
A subgroup $H \leqslant \ZZ^2$ is preserved by $\ell$ if and only if
$H$ is generated by a pair of elements of the form $(p,q),(-q,p)$.
In this case $|\ZZ^2/H|=p^2+q^2$. For a given integer $n$ the number
of invariant under $\ell$  subgroups $H$ of index $n$ in $\ZZ^2$ is
given by $\tau(n)$.
\end{lemma}
\begin{proof}
%{\bf New Version}
Suppose $H$ is generated by elements $(p,q)$ and
$(-q,p)=\ell\big((p,q)\big)$. Then obviously $\ell(H)=H$. Also
$|\ZZ^2/H|=p^2+q^2$, since $p^2+q^2$ is the number of integer points
in a fundamental domain of $H$.

Vice versa, suppose $\ell(H)=H$. Denote $d(x,y)=x^2+y^2$. Let
$u=(p,q)\in H\setminus\{0\}$ be an element with the minimal value of
$d(u)$. Consider the subgroup $H_1=\langle u,\ell(u) \rangle
\leqslant H$. Assume $H_1 \neq H$ and $v \in H\setminus H_1$. Since
$H_1=\langle u,\ell(u) \rangle$, the fundamental domain of $H_1$ is
a square with vertices $0,u,\ell(u),u+\ell(u)$. That means that the
plane splits into the squares of the form
$w,w+u,w+\ell(u),w+u+\ell(u), w \in H_1$. One of this squares
contains $v$. Note that the distance from a point inside a square to
one of its vertices is not greater then the side of this square.
This contradicts the minimality of $d(u)$, thus $H_1=H$.

%{\bf Old variant.}
% Note that the number of pairs $(p,q)$ with
%$p^2+q^2=n,\,p>0,\,q\ge0$ is $\tau(n)$. If elements $(p,q)$ and
%$(p',q')$ generate same group $H$, then $d(p-p',q-q')<n$, which
%contradicts minimality of $d(p,q)$. Thus different pairs lead to
%different subgroups $H$, that is the number subgroups is $\tau(n)$.
%
%{\bf New variant.}
To find the number of subgroups $H$ note that the
number of pairs $(p,q)$ with $p^2+q^2=n,\,p>0,\,q\ge0$ is $\tau(n)$.
As it was proven above, for each pair $(p,q)$ of the above type two
pairs $(p,q)$ and $\ell\big((p,q)\big)$ generate a subgroup $H$ of
the required type. Moreover $d(p,q)$ takes the minimal value among
$d(v),\,v\in H\setminus\{0\}$. Suppose two different pairs $(p,q)$
and $(p',q')$ correspond the same subgroup $H$. Then $(p-p',q-q')\in
H$ and $0<d(p-p',q-q')<d(p,q)$, which contradiction proves that
there is a one-to-one correspondence between pairs $(p,q)$ and
subgroups $H$.
\end{proof}

\begin{corollary}\label{number of quarters}
Let $n$ be an integer. Consider the set of all subgroups $H$ of
index $n$ in $\ZZ^2$, such that $\ell(H)= H$. By $R(n)$ denote the
sum over subgroups $H$ of the number of cosets $\nu \in \ZZ^2/H$
with $\ell(\nu)=\nu$. Then
$$
R(n)=\tau(n)+\tau(\frac{n}{2}).
$$
\end{corollary}

\begin{proof}
Consider subgroup $H \leqslant \ZZ^2$ with $|\ZZ^2:H|=n$ and
$\ell(H)=H$. \Cref{lemG4-1.2} claims that $H$ has a pair of
generators $(p,q),(-q,p)$, where $p^2+q^2=n$. Suppose
$\ell(\nu)=\nu$ holds for some coset $\nu \in \ZZ^2/H$. Let
$(a,b)\in \ZZ^2$ be a representative of coset $\nu$. Then
$\nu-\ell(\nu)=(a-b,a+b)\in i(p,q)+j(-q,p)$. That is
$(a,b)=i(\frac{p+q}{2},\frac{q-p}{2})+j(\frac{p-q}{2},\frac{p+q}{2})$
for some integer $i,j$. Then modulo $\langle (p,q),(-q,p) \rangle$
there are only two different choices for pairs $(a,b)$: one
corresponding to $i=j=0$, another corresponding to $i=1,\,j=0$. The
first pair is always integer, the second one is integer if and only
if $p\equiv q \mod 2$. Also, $p\equiv q \mod 2$ if and only if
$2\mid p^2+q^2=n$. That is, for a fixed $H$ there is one choice of
$\nu$ if $2\nmid n$ and two choices if $2\mid n$. By
\Cref{lemG4-1.2}, the number of possible subgroups $H$ is $\tau(n)$.
So $R(n)=\tau(n)$ if $2 \nmid n$ and $R(n)=2\tau(n)$ if $2 \mid n$.
Finally note $\tau(\frac{n}{2})=\tau(n)$ if $2 \mid n$ and
$\tau(\frac{n}{2})=0$ otherwise. Then we have the required
$R(n)=\tau(n)+\tau(\frac{n}{2})$.
%
%For a subgroup $H < \ZZ^2$ preserved by $\ell$ the number of $\nu
%\in \ZZ^2/H$ with $\ell(\nu)=\nu$ equals 1 if $3 \nmid |\ZZ^2:H|$
%and equals 3 if $3 \mid |\ZZ^2:H|$. The number for such subgroups
%$H$ is $\theta(n)$. Keeping in the mind that
%$\theta(\frac{n}{3})=\theta(n)$ if $3 \mid n$, and
%$\theta(\frac{n}{3})=0$ otherwise, one gets the required.
\end{proof}
\begin{remark}\label{number of quarters-remark}
For a fixed $H\leqslant \ZZ^2$ the number of $\nu \in \ZZ^2/H$ with
$\ell(\nu)=\nu$ coincides with $|\ZZ^2/\langle (1,-1), (1,1),
H\rangle|$. So
$$
R(n)=\sum_{H\leqslant \ZZ^2,\,|\ZZ^2/H|=n}|\{\nu| \nu\in
\ZZ^2/H,\,\ell(\nu)=\nu\}|=\sum_{H\leqslant \ZZ^2,\,|\ZZ^2/H|=n}
|\ZZ^2/\langle (1,-1), (1,1), H\rangle|.
$$
\end{remark}

\section{On the coverings of $\mathcal{G}_{2}$}\label{partG2}
\subsection{The structure of the group $\pi_{1}(\mathcal{G}_{2})$ }

The following proposition provides the canonical form of an element
in $\pi_{1}(\mathcal{G}_{2})$.

\begin{proposition}\label{propG2-1}
\begin{itemize}
\item[(i)] Each element of $\pi_{1}(\mathcal{G}_{2})$ can be represented in the canonical form $x^ay^bz^c$
for some integer $a,b,c$.
\item[(ii)] The product of two canonical forms is given by the
formula
\begin{equation}\label{multlawG2}
x^ay^bz^c \cdot x^{d}y^{e}z^{f}=x^{a+(-1)^cd}y^{b+(-1)^ce}z^{c+f}.
\end{equation}
\item[(iii)] The canonical epimorphism $\phi_{\mathcal{G}2}: \pi_{1}(\mathcal{G}_{2}) \to
\pi_{1}(\mathcal{G}_{2})/\langle x,y\rangle \cong \ZZ$, given by the
formula $x^ay^bz^c \to c$ is well-defined.
\item[(iv)] The representation in the canonical form $g=x^ay^bz^c$ for
each element $g\in \pi_{1}(\mathcal{G}_{2})$ is unique.
\end{itemize}
\end{proposition}

%{\bf New Version}
Routinely follows from the definition of the
group.
\smallskip
{\bf Notation.} By $\Gamma$ denote the subgroup of
$\pi_{1}(\mathcal{G}_{2})$ generated by $x,y$.
%Set $\Gamma=\langle x,y\rangle$.
\smallskip

In the next definition we introduce the invariants, similar to those
used in \Cref{number of sublattices}.

\begin{definition}\label{defG2-invariants}
Suppose all elements of $\pi_{1}(\mathcal{G}_{2})$ are represented
in the canonical form. Let $\Delta$ be a subgroup of finite index
$n$ in $\pi_{1}(\mathcal{G}_{2})$. Put $H(\Delta)=\Delta\bigcap
\Gamma$. By $a(\Delta)$ denote the minimal positive exponent at $z$
among all the elements of $\Delta$. Choose an element $Z(\Delta)$
with such exponent at $z$, represented in the form
$Z(\Delta)=hz^{a(\Delta)}$, where $h \in \Gamma$. By
$\nu(\Delta)=hH(\Delta)$ denote the coset in coset decomposition
$\Gamma/H(\Delta)$.
%For the coset decomposition $\Gamma/H(\Delta)$
%we will use the additive notation.
By $Y(\Delta)$ and $X(\Delta)$
denote a pair of generators of $H(\Delta)$ of the form, provided by
\Cref{number of sublattices}, that is
$Y(\Delta)=x^{e(\Delta)}y^{b(\Delta)}$, $X(\Delta)=x^{c(\Delta)}$
where $0\le e(\Delta) < c(\Delta)$. Further we will omit $\Delta$
for $X(\Delta),Y(\Delta),Z(\Delta)$.
\end{definition}

Note that the invariants $a(\Delta)$, $H(\Delta)$ and $\nu(\Delta)$
are well-defined. In particular, the latter one does not depends on
a choice of $Z(\Delta)$. Also
%$a(\Delta)b(\Delta)c(\Delta)=a(\Delta)[\Gamma:H(\Delta)]=[\pi_{1}(\mathcal{G}_{2}):\Delta]$.
$a(\Delta)[\Gamma:H(\Delta)]=[\pi_{1}(\mathcal{G}_{2}):\Delta]$.

%{\bf End of variants.}

\begin{definition}\label{defG2-3-plet}
A 3-plet $(a,H,\nu)$ is called {\em $n$-essential} if the following
conditions holds:
\begin{itemize}
\item[(i)] $a$ is a positive divisor of $n$,
\item[(ii)] $H$ is a subgroup of index $n/a$ in $\Gamma$,
%and $[\ZZ^2 : \phi(\Delta)\big]=n/l(\Delta)$,
\item[(iii)] $\nu$ is an element of $\Gamma/H$.
\end{itemize}
\end{definition}

\begin{lemma}\label{lemG2-2}
For arbitrary $n$-essential 3-plet $(a,H,\nu)$ there exists a
subgroup $\Delta$ in the group $\pi_1(\mathcal{G}_2)$ such that
$(a,H,\nu)=(a(\Delta),H(\Delta),\nu(\Delta))$.
\end{lemma}
\begin{proof}
Let $h$ be a representative of the coset $\nu$. In case $a$ is odd
consider the set
$$
 \{hz^{(2l+1)a}H| l\in \ZZ\}\bigcup
\{z^{2la}H|l\in \ZZ\}.
$$
This set is a subgroup of index $n$ in $\pi_1(\mathcal{G}_2)$, this
can be proven directly.

Similarly, in case $a$ is even the set
$$
 \{h^lz^{la}H|l\in \ZZ\}
$$
form a subgroup of index $n$ in $\pi_1(\mathcal{G}_2)$.
\end{proof}

\begin{proposition}\label{propG2-2}
There is a bijection between the set of $n$-essential 3-plets
$(a,H,\nu)$ and the set of subgroups $\Delta$ of index $n$ in
$\pi_1(\mathcal{G}_2)$ given by the correspondence $\Delta
\leftrightarrow (a(\Delta),H(\Delta),\nu(\Delta))$. Moreover,
$\Delta\cong \pi_{1}(\mathcal{G}_{1})$  if $a(\Delta)$ is even and
$\Delta\cong \pi_{1}(\mathcal{G}_{2})$ if $a(\Delta)$ is odd.
\end{proposition}

\begin{proof}
Consider the set of subgroups $\Delta$ of index $n$ in
$\pi_1(\mathcal{G}_2)$. \Cref{defG2-invariants} provides the map of
the family of subgroups $\Delta$ to $n$-essential 3-plets.
\Cref{lemG2-2} shows that this map is a bijection. Now we describe
the isomorphism type of a subgroup.

If $a(\Delta)$ is even \Cref{lemG2-2} implies that $\Delta$ is a
subgroup of $\langle x,y,z^2\rangle$. Substituting the canonical
representations with even exponent at $z$ into (\ref{multlawG2}) one
gets that $\langle x,y,z^2\rangle \cong \ZZ^3$, thus $\Delta$ is a
subgroup of finite index in $\ZZ^3$. As a result, $\Delta$ is
isomorphic to $\ZZ^3$.

Consider the case $a(\Delta)$ is odd. For the sake of brevity, we
write $X=X(\Delta)$, $Y=Y(\Delta)$ and $Z=Z(\Delta)$. By
\Cref{defG2-invariants} subgroup $\Delta$ is generated by $X,Y,Z$.
Direct verification shows that the relations $XYX^{-1}Y^{-1}=1$,
$ZXZ^{-1}= X^{-1}$, and $ZYZ^{-1}=Y^{-1}$ hold. Further we call this
relations {\em the proper relations of the subgroup $\Delta$}. Thus
the map $x \to X, \,\, y \to Y,\,\, z \to Z$ can be extended to an
epimorphism $\pi_1(\mathcal{G}_2) \to \Delta $. To prove the
epimorphism is really an isomorphism we need to show that each
relation in $\Delta$ is a corollary of proper relations. We call a
relation, that is not a corollary of proper relations an {\em
improper relation}.

Assume the contrary, i.e. there are some improper relations in
$\Delta$. Since in $\Delta$ the proper relations holds, each element
can be represented in the canonical form, given by \Cref{propG2-1}
in terms of $X,Y,Z$, by using just the proper relations. That is
each element $g$ can be represented as
$$
g=X^{r}Y^{s}Z^{t}.
$$
If there is an improper relation then there is an equality
\begin{equation}\label{absurdum1-G2}
X^{r}Y^{s}Z^{t}=X^{r'}Y^{s'}Z^{t'},
\end{equation}
where at least one of the inequalities $r\neq r'$, $s\neq s'$,
$t\neq t'$ holds. Applying $\phi_{\mathcal{G}2}$ to both parts we
get $ta(\Delta)=t'a(\Delta)$, thus $t=t'$. Then
$X^{r}Y^{s}=X^{r'}Y^{s'}$, that means
\begin{equation}\label{absurdum2-G2}
\left\{
\begin{aligned}
c(\Delta)r+e(\Delta)s=c(\Delta)r'+e(\Delta)s'\\
b(\Delta)s=b(\Delta)s'\\
\end{aligned} \right.
\end{equation}

Keep in mind that $b(\Delta)c(\Delta)\neq 0$ since
$a(\Delta)b(\Delta)c(\Delta)=n$. The contradiction of equations
(\ref{absurdum2-G2}) with $(r,s)\neq(r',s')$ proves that $\Delta
\cong \pi_1(\mathcal{G}_2) $.
\end{proof}

\subsection{The proof of \Cref{th-1-dicosm}}

Proceed to the proof of \Cref{th-1-dicosm}. \Cref{propG2-2} claims
that each subgroup $\Delta$ of finite index $n$ is isomorphic to
 $\pi_1(\mathcal{G}_{2})$ or $\ZZ^3$, depending upon whether
$a(\Delta)$ is odd or even. Consider these two cases separately.

{\bf Case (i).} To find the number of subgroups isomorphic to
$\pi_1(\mathcal{G}_{2})$, by Proposition \ref{propG2-2} we need to
calculate the cardinality of the set of $n$-essential 3-plets with
odd $a$.

For each odd $a\mid n$ there are $\sigma_1(\frac{n}{a})$ subgroups
$H$ in $\Gamma$, such that $\big|\Gamma:H\big|=\frac{n}{a}$. Also
there are $\frac{n}{a}$ different choices of a coset $\nu$. Thus,
for each odd $a$ the number of $n$-essential 3-plets is
$\frac{n}{a}\sigma_1(\frac{n}{a})$. So, the total number of
subgroups is given by
$$
s_{\pi_1(\mathcal{G}_{2}), \pi_1(\mathcal{G}_{2})}(n) =  \sum_{a
\mid n,\, 2\nmid a}  \frac{n}{a}\sigma_1(\frac{n}{a}).
$$
Equivalently,
$$
s_{\pi_1(\mathcal{G}_{2}), \pi_1(\mathcal{G}_{2})}(n) =  \sum_{a
\mid n} \frac{n}{a}\sigma_1(\frac{n}{a}) - \sum_{2a \mid n}
\frac{n}{2a}\sigma_1(\frac{n}{2a})=\omega(n)-\omega(\frac{n}{2}).
$$

{\bf Case (ii).} Similarly to the previous case, we get the formula
$$
s_{\ZZ^3, \pi_1(\mathcal{G}_{2})}(n) =  \sum_{2a \mid n}
\frac{n}{2a}\sigma_1(\frac{n}{2a})=\omega(\frac{n}{2}).
$$

\subsection{The proof of \Cref{th-2-dicosm}}

The isomorphism types of subgroups are already provided by
\Cref{propG2-2}. Thus we just have to calculate the number of
conjugacy classes for each type.

\smallskip
{\bf Notation.} By $\xi$ denote the canonical homomorphism $\Gamma
\to \Gamma/H(\Delta)$.
\smallskip

To prove case (ii) consider a subgroup $\Delta$ of index $n$ in
$\pi_1(\mathcal{G}_{2})$ isomorphic to $\ZZ^3$. By \Cref{propG2-2},
a subgroup $\Delta$ is uniquely defined by the $n$-essential 3-plet
$(a(\Delta),H(\Delta),\nu(\Delta))$, where $a$ is even. Consider the
conjugacy class of subgroups $\Delta^g,\,g \in
\pi_{1}(\mathcal{G}_{2})$. It consists of subgroups corresponding to
3-plets $\big(a(\Delta^g),H(\Delta^g),\nu(\Delta^g)\big),\,g \in
\pi_1(\mathcal{G}_{2})$.

Obviously, $a(\Delta^g)=a(\Delta)$ and $H(\Delta^g)=H(\Delta)$.
Furthermore,
$\nu(\Delta)=\nu(\Delta^x)=\nu(\Delta^y)=\nu(\Delta^{z^2})$ and
$\nu(\Delta)=-\nu(\Delta^z)$, thus $\nu(\Delta^g)=\pm\nu(\Delta)$.
So, the conjugacy class of the subgroup $\Delta$ consists of one or
two subgroups, depending on whether the condition $2\nu(\Delta)=0$
holds or not. In the former case, $\Delta$ is normal in
$\pi_{1}(\mathcal{G}_{2})$.

By $\mathfrak{M}_1$ and $\mathfrak{M}_2$ denote the set of normal
subgroups, and the set of subgroups having exactly two subgroups in
their conjugacy class, respectively. The obvious formula for the
number of conjugacy classes is
$c_{\ZZ^3,\pi_1(\mathcal{G}_2)}=|\mathfrak{M}_1|+\frac{|\mathfrak{M}_2|}{2}$,
we rewrite it in the form
$c_{\ZZ^3,\pi_1(\mathcal{G}_3)}=\frac{|M_1|}{2}+\frac{|\mathfrak{M}_1|+|\mathfrak{M}_2|}{2}$.
Note that $ \mathfrak{M}_1 \bigcup \mathfrak{M}_2$ is the set of all
subgroups, thus $|\mathfrak{M}_1|+|\mathfrak{M}_2|$ is given by
\Cref{th-1-dicosm}. The cardinality of the subset of
$\mathfrak{M}_1$ with the fixed value of $a(\Delta)$ is exactly
$S(\frac{n}{a(\Delta)})$, which was defined in \Cref{number of
halfes}, so
$S(\frac{n}{a(\Delta)})=\sigma_1((\frac{n}{a(\Delta)}))+3\sigma_1((\frac{n}{2a(\Delta)}))$.

Summing over all possible values of $a(\Delta)$ one gets
$$
c_{\ZZ^3,\pi_{1}(\mathcal{G}_{2})}(n) =\frac{1}{2}\Big(\sum_{a\mid
n,\,2\mid
a}\Big(\sigma_1(\frac{n}{a})+3\sigma_1(\frac{n}{2a})\Big)+\omega(\frac{n}{2})\Big)=
\frac{1}{2}\Big(\sigma_2(\frac{n}{2})+3\sigma_2(\frac{n}{4})+\omega(\frac{n}{2})\Big).
$$
This concludes the proof of case (ii) of \Cref{th-2-dicosm}.

The proof of case (i) resembles the proof of (ii). Consider a
subgroup $\Delta$ of index $n$ in $\pi_1(\mathcal{G}_{2})$
isomorphic to $\pi_1(\mathcal{G}_{2})$. By \Cref{propG2-2}, the
subgroup $\Delta$ is uniquely defined by the $n$-essential 3-plet
$(a(\Delta),H(\Delta),\nu(\Delta))$ with odd $a(\Delta)$. Consider
the conjugacy class of subgroups $\Delta^g, \,g \in
\pi_{1}(\mathcal{G}_{2})$. It consists of subgroups, bijectively
corresponding to 3-plets $(a(\Delta^g),H(\Delta^g),\nu(\Delta^g))$.

Obviously, $a(\Delta^g)=a(\Delta)$ and $H(\Delta^g)=H(\Delta)$.
Furthermore, $\nu(\Delta^x)=\nu(\Delta)+2\xi(x)$,
$\nu(\Delta^y)=\nu(\Delta)+2\xi(y)$ and
$\nu(\Delta)=-\nu(\Delta^z)$. So,
$\nu(\Delta^g)=\nu(\Delta)+2r\xi(x)+2s\xi(y)$ for integer $r$ and
$s$.

Thus a conjugacy class is uniquely defined by $a(\Delta)$,
$H(\Delta)$ and the coset of $\nu(\Delta)$ in $\Gamma/\langle
H,2x,2y \rangle$. By \Cref{number of halfes-remark}, for a fixed
$a(\Delta)$ the sum over all possible $H$ of the numbers
$|\Gamma:\langle H,2x,2y) \rangle|$ is
$S(\frac{n}{a(\Delta)})=\sigma_1(\frac{n}{a(\Delta)})+3\sigma_1(\frac{n}{2a(\Delta)})$.
Summing this over all possible values of $a$ which are odd divisors
of $n$, we get the final formula
$$
c_{\pi_{1}(\mathcal{G}_{2}),\pi_{1}(\mathcal{G}_{2})}(n) = \sum_{a
\mid n\; 2\nmid
a}\big(\sigma_1(\frac{n}{a})+3\sigma_1(\frac{n}{2a})\big) =
\sigma_2(n) + 2\sigma_2(\frac{n}{2}) - 3\sigma_2(\frac{n}{4}).
$$

\section{On the coverings of $\mathcal{G}_{4}$}\label{partG4}
\subsection{The structure of the group $\pi_{1}(\mathcal{G}_{4})$ }
Recall that $\pi_{1}(\mathcal{G}_{4})$ is given by generators and
relations by $\pi_{1}(\mathcal{G}_{4})=\langle \tilde{x}, \tilde{y},
\tilde{z}: \tilde{x}\tilde{y}\tilde{x}^{-1}\tilde{y}^{-1}=1,
\tilde{z}\tilde{x}\tilde{z}^{-1}= \tilde{y},
\tilde{z}\tilde{y}\tilde{z}^{-1}=\tilde{x}^{-1}
 \rangle$ The following proposition provides the canonical
form of an element in $\pi_{1}(\mathcal{G}_{4})$.

\begin{proposition}\label{propG4-1}
\begin{itemize}
\item[(i)] Each element of $\pi_{1}(\mathcal{G}_{4})$ can be represented in the canonical form $\tilde{x}^a\tilde{y}^b\tilde{z}^c$
for some integer $a,b,c$.
\item[(ii)] The product of two canonical forms is given by the
formula
\begin{equation}\label{multlawG4}
\tilde{x}^a\tilde{y}^b\tilde{z}^c \cdot
\tilde{x}^{d}\tilde{y}^{e}\tilde{z}^{f}= \left\{
\begin{aligned}
\tilde{x}^{a+d}\tilde{y}^{b+e}\tilde{z}^{c+f} \quad \text{if} \quad c\equiv 0\mod4 \\
\tilde{x}^{a-e}\tilde{y}^{b+d}\tilde{z}^{c+f} \quad \text{if} \quad c\equiv 1\mod4 \\
\tilde{x}^{a-d}\tilde{y}^{b-e}\tilde{z}^{c+f} \quad \text{if} \quad c\equiv 2\mod4 \\
\tilde{x}^{a+e}\tilde{y}^{b-d}\tilde{z}^{c+f} \quad \text{if} \quad c\equiv 3\mod4 \\
\end{aligned} \right.
\end{equation}
\item[(iii)] The canonical epimorphism $\phi_{\mathcal{G}4}: \pi_{1}(\mathcal{G}_{4}) \to
\pi_{1}(\mathcal{G}_{4})/\langle \tilde{x},\tilde{y}\rangle \cong
\ZZ$, given by the formula $\tilde{x}^a\tilde{y}^b\tilde{z}^c \to c$
is well-defined.
\item[(iv)] The representation in the canonical form $g=\tilde{x}^a\tilde{y}^b\tilde{z}^c$ for
each element $g\in \pi_{1}(\mathcal{G}_{2})$ is unique.
\end{itemize}
\end{proposition}
\begin{proof}
The proof is similar to the proof of \Cref{propG2-1}.
\end{proof}

\smallskip
{\bf Notation.} By $\Gamma$ denote the subgroup of
$\pi_{1}(\mathcal{G}_{4})$ generated by $\tilde{x},\tilde{y}$.
%Set $\Gamma=\langle x,y\rangle$.
\smallskip

In the next definition we introduce the invariants, similar to those
used in \Cref{number of sublattices}.

\begin{definition}\label{defG4-invariants}
Suppose all elements of $\pi_{1}(\mathcal{G}_{4})$ are represented
in the canonical form. Let $\Delta$ be a subgroup of finite index
$n$ in $\pi_{1}(\mathcal{G}_{4})$. Put $H(\Delta)=\Delta\bigcap
\Gamma$. By $a(\Delta)$ denote the minimal positive exponent at $z$
among all the elements of $\Delta$. Choose an element $Z(\Delta)$
with such exponent at $z$, represented in the form
$Z(\Delta)=hz^{a(\Delta)}$, where $h \in \Gamma$. By
$\nu(\Delta)=hH(\Delta)$ denote the coset in coset decomposition
$\Gamma/H(\Delta)$.
%For the coset decomposition $\Gamma/H(\Delta)$
%we will use the additive notation.
%By $Y(\Delta)$ and $X(\Delta)$
\end{definition}

Note that the invariants $a(\Delta)$, $H(\Delta)$ and $\nu(\Delta)$
are well-defined. In particular, the latter one does not depends on
a choice of $Z(\Delta)$. Also
$a(\Delta)[\Gamma:H(\Delta)]=[\pi_{1}(\mathcal{G}_{4}):\Delta]$.

%It is worth noting that, despite the choice of $Z$ and $h$ is not
%unique, the choice of $\nu(\Delta)$ is.

%\begin{lemma}\label{lemG4-1}
%The number $a(\Delta)$, the subgroup $H(\Delta)$ and the coset
%$\nu(\Delta)$ are well-defined. Furthermore
%$a(\Delta)[\Gamma:H(\Delta)]=[\pi_{1}(\mathcal{G}_{4}):\Delta]$.
%\end{lemma}
%
%The proof is similar to \Cref{lemG2-1}.

%{\bf ATTENTION!!! Привести к единообразию с прошлым разделом.}

\begin{lemma}\label{lemG4-1.1}
If $a(\Delta)$ is odd then $H(\Delta)\lhd \pi_1(\mathcal{G}_4)$.
\end{lemma}

\begin{proof}
%Denote $Z=\tilde{x}^s\tilde{y}^t\tilde{z}^{a(\Delta)} \in \Delta$.
Recall that $Z=h\tilde{z}^{a(\Delta)}\in \Delta$, where $a(\Delta)$
is odd and $h \in \langle \tilde{x}, \tilde{y}\rangle$. First
$H(\Delta)^Z=H(\Delta)$. Also
$H(\Delta)^{\tilde{x}}=H(\Delta)^{\tilde{y}}=H(\Delta)^{\tilde{z}^2}=H(\Delta)$.
The former fact means that $H(\Delta)^g=H(\Delta),\; g\in
\pi_1(\mathcal{G}_4)$, hence $\langle
\tilde{x},\tilde{y},\tilde{z}^2,Z \rangle=\pi_1(\mathcal{G}_4)$ in
the case of odd $a(\Delta)$.
\end{proof}

%\begin{lemma}\label{lemG4-1.2}
%Let $G$ be a subgroup in $\Gamma$. Then $G\lhd \pi_1(\mathcal{G}_4)$
%if and only if there exist a pair of generators of $G$ of the form
%$(\tilde{x}^p\tilde{y}^q,\tilde{x}^{-q}\tilde{y}^p)$. In this case
%$[\Gamma:G]=p^2+q^2$.
%\end{lemma}
%
%\begin{proof}
%Obvious.
%\end{proof}
%
%\begin{lemma}\label{lemG4-1.3}
%The number of subgroups of index $m$ in $\Gamma$ normal in
%$\pi_1(\mathcal{G}_4)$ is $\tau(m)$, where $\tau(m)=|\{(s,t)|s,t \in
%\ZZ, s>0, t\ge 0, s^2+t^2=m\}|$.
%\end{lemma}
%
%\begin{proof}
%The trivial corollary of the previous lemma.
%\end{proof}

\begin{definition}\label{defG4-3-plet}
A 3-plet $(a,H,\nu)$ is called {\em $n$-essential} if the following
conditions holds:
\begin{itemize}
\item[(i)] $a$ is a positive divisor of $n$,
\item[(ii)] $H$ is a subgroup of index $n/a$ in $\Gamma$ also if $a$ is odd then $H\lhd
\pi_1(\mathcal{G}_4)$,
\item[(iii)] $\nu$ is an element of $\Gamma/H$.
\end{itemize}
\end{definition}

\begin{lemma}\label{lemG4-2}
For an arbitrary $n$-essential 3-plet $(a,H,\nu)$ there exists a
subgroup $\Delta$ of $\pi_1(\mathcal{G}_4)$ such that
$(a,H,\nu)=(a(\Delta),H(\Delta),\nu(\Delta))$.
\end{lemma}

\begin{proof}
Consider some $n$-essential 3-plet $(a,H,\nu)$ and let $h$ be a
representative of the coset $\nu$.

If $a\equiv 0 \mod 4$ consider the set
$$
\Delta = \{h^l\tilde{z}^{la}H|l\in \ZZ\}.
$$
Direct verification shows that $\Delta$ is a subgroup of index $n$
in $\pi_1(\mathcal{G}_4)$ with required characteristics $a(\Delta)$,
$H(\Delta)$ and $\nu(\Delta)$.

In the three remained cases the structure of subgroup is defined the
following way: in case $a\equiv 2 \mod 4$ put
$$
\Delta = \{h\tilde{z}^{(2l+1)a}H|l\in \ZZ\}\bigcup
\{\tilde{z}^{2la}H|l\in \ZZ\}.
$$

In case $a\equiv 1 \mod 4$ put
$$
\aligned & \Delta = \{h\tilde{z}^{(4l+1)a}H|l\in \ZZ\}\bigcup
\{hh^{\tilde{z}}\tilde{z}^{(4l+2)a}H|l\in \ZZ\}\bigcup\\&
\{h^{\tilde{z}}\tilde{z}^{(4l+3)a}H|l\in \ZZ\}\bigcup
\{\tilde{z}^{4la(\Delta)}H|l\in \ZZ\},
\endaligned
$$
keep in mind that $gg^{\tilde{z}^2}=1$ for all $g \in \Gamma$.

In the same way, if $k\equiv 3 \mod 4$ put
$$
\aligned & \Delta = \{h\tilde{z}^{(4l+1)k}H|l\in \ZZ\}\bigcup
\{hh^{\tilde{z}^3}\tilde{z}^{4lk+2}H|l\in \ZZ\}\bigcup\\&
\{h^{\tilde{z}^3}\tilde{z}^{(4l+3)k}H|l\in \ZZ\}\bigcup
\{\tilde{z}^{4lk(\Delta)}H|l\in \ZZ\}.
\endaligned
$$
\end{proof}

\begin{proposition}\label{propG4-2}
There is a bijection between the set of $n$-essential 3-plets
$(a,H,\nu)$ and the set of subgroups of index $n$ in
$\pi_1(\mathcal{G}_4)$ given by the correspondence $\Delta
\leftrightarrow (a(\Delta),H(\Delta),\nu(\Delta))$. Moreover,
$\Delta\cong \ZZ^3$ if $a(\Delta) \equiv 0 \mod 4$, $\Delta\cong
\pi_1(\mathcal{G}_2)$ if $a(\Delta) \equiv 2 \mod 4$ and
$\Delta\cong \pi_1(\mathcal{G}_4)$ if $a(\Delta) \equiv 1 \mod 2$.
\end{proposition}

\begin{proof}
Consider the set of subgroups $\Delta$ of index $n$ in
$\pi_1(\mathcal{G}_4)$. \Cref{defG4-invariants} together with
\Cref{lemG4-1.1} build the map of the family of subgroups
 $\Delta$ to $n$-essential 3-plets $\Delta
\leftrightarrow (a(\Delta),H(\Delta),\nu(\Delta))$. \Cref{lemG4-2}
shows that this map is a bijection.

The proof of isomorphism part of the statement is similar to
\Cref{propG2-2}.
%Consider canonical forms of all elements in
%$\Delta$.
In case $a(\Delta) \equiv 0 \mod 4$ equation \ref{multlawG4}
provides the commutativity, thus $\Delta\cong \ZZ^3$. The case
$a(\Delta) \equiv 2 \mod 4$ is proven in \Cref{propG2-2}. The case
$a(\Delta) \equiv 1 \mod 2$ follows the same way as the proof of
\Cref{propG2-2}: one fixes the suitable generators
$\widetilde{X},\widetilde{Y},\widetilde{Z}$, for which all required
relations hold, and prove that any unnecessary relation implies a
relation in $\Gamma$, which contradiction completes the proof.
\end{proof}

\subsection{The proof of \Cref{th-1-tetracosm}}
The isomorphism types of finite index subgroups in
$\pi_1(\mathcal{G}_{4})$ are provided by \Cref{propG4-2}. The cases
of subgroups $\pi_1(\mathcal{G}_2)$ and $\ZZ^3$ in the group
$\pi_1(\mathcal{G}_4)$ are similar to the cases of subgroups
$\pi_1(\mathcal{G}_2)$ and $\ZZ^3$ in the group
$\pi_1(\mathcal{G}_2)$ respectively.

In case of a subgroup $\Delta$ isomorphic to $\pi_1(\mathcal{G}_4)$,
for each fixed $k \equiv 1 \mod 2$ there are $\tau(\frac{n}{k})$
different $H$, and $\frac{n}{k}$ different $\nu$ for each fixed $H$,
thus the final value is
$$
s_{\pi_1(\mathcal{G}_{4}), \pi_1(\mathcal{G}_{4})}(n) =  \sum_{k
\mid n} \frac{n}{k}\tau(\frac{n}{k}) - \sum_{2k \mid n}
\frac{n}{2k}\tau(\frac{n}{2k})=\sum_{k \mid n}k\tau(k)-\sum_{k \mid
\frac{n}{2}}k\tau(k).
$$

\subsection{The proof of \Cref{th-2-tetracosm}}

The isomorphism types of subgroups are already provided by
\Cref{propG4-2}. Thus we just have to calculate the number of
conjugacy classes for each type. Consider the cases in order (iii),
(ii), (i).

\subsubsection{Case (iii)}
Let $\Delta$ be a subgroup of index $n$ in $\pi_1(\mathcal{G}_4)$,
isomorphic to $\ZZ^3$.

Since
$\Delta=\Delta^{\tilde{x}}=\Delta^{\tilde{y}}=\Delta^{\tilde{z}^4}$,
one of the following three possibilities holds:
\begin{itemize}
\item $\Delta$ is normal if $\Delta^{\tilde{z}}=\Delta$
\item conjugacy class of $\Delta$ contains exactly 2 subgroups if $\Delta^{\tilde{z}}\neq\Delta$
but $\Delta^{\tilde{z}^2}=\Delta$
\item conjugacy class of $\Delta$ contains exactly 4 subgroups if
$\Delta^{\tilde{z}^2}\neq\Delta$.
\end{itemize}
%\begin{lemma}\label{lemG4-3} The conjugacy class of $\Delta$ consists of 1, 2 or 4 subgroups.
%
%\end{lemma}
%
%\begin{proof}
%In virtue of \Cref{propG4-2} the subgroup $\Delta$ is uniquely
%determined by its $n$-essential 3-plet. Also, by \Cref{propG4-2} if
%$\tilde{x}^a\tilde{y}^b\tilde{z}^c \in \Delta$ then $4\mid c$.
%
%Again, $a(\Delta^g)=a(\Delta), \, g \in \pi_1(\mathcal{G}_{4})$ and
%$H(\Delta^{\tilde{x}})=H(\Delta^{\tilde{y}})=H(\Delta^{\tilde{z}^2})=H(\Delta)$.
%Also $\nu(\Delta)$:
%$\nu(\Delta^{\tilde{x}})=\nu(\Delta^{\tilde{y}})=\nu(\Delta^{\tilde{z}^4})=\nu(\Delta)$,
%here it is important that $4\mid c$.
%
%Thus the conjugacy class of $\Delta$ contains at most 4 groups:
%$\Delta$, $\Delta^{\tilde{z}}$, $\Delta^{\tilde{z}^2}$ and
%$\Delta^{\tilde{z}^3}$. If $\Delta \neq \Delta^{\tilde{z}^2}$ then
%it contains exactly 4 groups, otherwise it contains two or one. %The
%%latter means that $\Delta$ is normal.
%In the latter case $\Delta$ is normal in $\pi_1(\mathcal{G}_{4})$.
%\end{proof}

\begin{definition}
By $\mathfrak{M}_1$, $\mathfrak{M}_2$ and $\mathfrak{M}_4$ denote
the respective sets of subgroups $\Delta \cong \ZZ^3$ of index $n$:
which are normal in $\pi_1(\mathcal{G}_4)$, which belong to a
conjugacy class of exactly two subgroups, and which belong to a
conjugacy class of exactly four subgroups.
\end{definition}

The obvious formula for the number of conjugacy classes is
$c_{\ZZ^3,\pi_1(\mathcal{G}_4)}=|\mathfrak{M}_1|+\frac{|\mathfrak{M}_2|}{2}+\frac{|\mathfrak{M}_4|}{4}$,
we rewrite it in the form
$c_{\ZZ^3,\pi_1(\mathcal{G}_4)}=\frac{|\mathfrak{M}_1|}{2}+\frac{|\mathfrak{M}_1|+|\mathfrak{M}_2|}{4}+\frac{|\mathfrak{M}_1|+|\mathfrak{M}_2|+|\mathfrak{M}_4|}{4}$.
%Note that $ \mathfrak{M}_1 \bigcup \mathfrak{M}_2 \bigcup
%\mathfrak{M}_4$ is the set of all subgroups, and $\mathfrak{M}_1
%\bigcup \mathfrak{M}_2$ is the set of all subgroups, which conjugacy
%class contains at most two subgroups.

\Cref{th-1-tetracosm} claims
$|\mathfrak{M}_1|+|\mathfrak{M}_2|+|\mathfrak{M}_4|=\omega(\frac{n}{4})$.

\begin{lemma}\label{lemG4-4}
$$|\mathfrak{M}_1|+|\mathfrak{M}_2|=\sigma_2(\frac{n}{4})+3\sigma_2(\frac{n}{8}).$$
\end{lemma}

\begin{proof}
We have to calculate the amount of subgroups $\Delta \cong \ZZ^3$ of
index $n$, satisfying $\Delta=\Delta^{\widetilde{z}^2}$, that is
$2\nu(\Delta)=0$. \Cref{number of halfes} provides this number for a
fixed value of $a(\Delta)$, summing over all possible $a(\Delta)$
gives the required.
\end{proof}

\begin{lemma}\label{lemG4-5}
$$|\mathfrak{M}_1|=\sum_{4a \mid n}\tau(\frac{n}{4a})+\sum_{8a \mid n}\tau(\frac{n}{8a}).$$
\end{lemma}

\begin{proof}
We have to calculate the amount of subgroups $\Delta \cong \ZZ^3$ of
index $n$ and such that $\Delta=\Delta^{\widetilde{z}}$, that is
$H(\Delta)^{\widetilde{z}}=H(\Delta)$ and
$\nu(\Delta)^{\widetilde{z}}=\nu(\Delta)$. Keep in mind that
$g^{\widetilde{z}}=\ell(g)$ for $g \in \Gamma$ (recall $\ell: \ZZ^2
\to \ZZ^2$ was introduced in \Cref{ell}). Thus \Cref{number of
quarters} provides the number of pairs $(H,\nu)$ for a fixed $a$.
Also \Cref{propG4-2} claims $4\mid a(\Delta)$. Summing
$S(\frac{n}{a})$ over all possible values of $a$ one gets required.
% Since
%$H(\Delta)=H(\Delta^{\widetilde{z}})$, $H(\Delta) \lhd
%\pi_1(\mathcal{G}_4)$. Then by \Cref{lemG4-1.3} the number of
%choices of $H(\Delta)$ is $\tau(\frac{n}{a(\Delta)})$. Finally,
%$\nu(\Delta)=\nu(\Delta^{\widetilde{z}})$, which is possible for one
%value if $\frac{n}{a(\Delta)}$ is odd and for two values if
%$\frac{n}{a(\Delta)}$ is even. Since
%$\tau(\frac{n}{2a(\Delta)})=\tau(\frac{n}{a(\Delta)})$ if
%$\frac{n}{a(\Delta)}$ is even, and $\tau(\frac{n}{2a(\Delta)})=0$
%otherwise, the number of pairs $\big(H(\Delta),\nu(\Delta)\big)$
%equals $\tau(\frac{n}{a(\Delta)})+\tau(\frac{n}{2a(\Delta)})$. We
%finish the prove summing the respective number of pairs over all
%possible values of $a(\Delta)$. Keep in mind that $4\mid a(\Delta)$,
%so
%$$
%|\mathfrak{M}_1|=\sum_{4a \mid n}\tau(\frac{n}{4a})+\sum_{8a \mid
%n}\tau(\frac{n}{8a}).
%$$
\end{proof}

Summarizing the results of \Cref{th-1-tetracosm} (iii),
\Cref{lemG4-4} and \Cref{lemG4-5} one gets
$$
\begin{aligned} &
c_{\ZZ^3,\pi_1(\mathcal{G}_4)}=\frac{1}{2}\sum_{4a \mid
n}\tau(\frac{n}{4a})+\frac{1}{2}\sum_{8a \mid
n}\tau(\frac{n}{8a})+\frac{1}{4}\sigma_2(\frac{n}{4})+\frac{3}{4}\sigma_2(\frac{n}{8})+\frac{1}{4}\omega(\frac{n}{4})=\\&
\frac{1}{4}\Big(\omega(\frac{n}{4})+\sigma_2(\frac{n}{4})+3\sigma_2(\frac{n}{8})+2\sum_{a\mid\frac{n}{4}}\tau(a)+2\sum_{a\mid\frac{n}{8}}\tau(a)\Big).
\end{aligned}
$$

\subsubsection{Case (ii)}
%{\bf Old Variant}
Let $\Delta$ be a subgroup of index $n$ in
$\pi_1(\mathcal{G}_4)$ isomorphic to $\pi_1(\mathcal{G}_2)$.

\smallskip
{\bf Notation.} By $\xi$ denote the canonical homomorphism $\Gamma
\to \Gamma/H(\Delta)$.
\smallskip

Recall that the subgroup is uniquely defined by an $n$-essential
3-plet. First we have to describe the triplets of all subgroups,
which belongs to the conjugacy class of $\Delta$. Again,
$a(\Delta^g)=a(\Delta),\,g \in \pi_1(\mathcal{G}_4)$. Also,
$H(\Delta^{\tilde{x}})=H(\Delta^{\tilde{y}})=H(\Delta^{\tilde{z}^2})=H(\Delta)$,
thus for arbitrary $g \in \pi_1(\mathcal{G}_4)$ either
$H(\Delta^g)=H(\Delta)$ or $H(\Delta^g)=H(\Delta^{\tilde{z}})$.
$\nu(\Delta^{\tilde{x}})=\nu(\Delta)+2\xi(\tilde{x})$,
$\nu(\Delta^{\tilde{y}})=\nu(\Delta)+2\xi(\tilde{y})$,
$\nu(\Delta^{\tilde{z}^2})=-\nu(\Delta)$. Then the conjugacy class
of $\Delta$ consists of all subgroups, corresponding to 3-plets
$\big(a(\Delta), H(\Delta),
\nu(\Delta)+\langle2\xi(\tilde{x}),2\xi(\tilde{y})\rangle\big)$ and
$\big(a(\Delta), H(\Delta^{\tilde{z}}),
\nu(\Delta^{\tilde{z}})+\langle2\xi(\tilde{x}),2\xi(\tilde{y})\rangle\big)$.

So, to calculate the number of conjugacy classes we need to
calculate the number of pairs consisting of a subgroup $H$ and an
element of $\Gamma/(\langle\tilde{x}^2,\tilde{y}^2,H\rangle$.

Fix some $a(\Delta)=a$. Each conjugacy class corresponds to two
pairs of a subgroup and an element of the factor-group, unless this
two pairs coincide. Analogous to the previous case, let
$\mathfrak{L}_1(a)$ be the family of defined above pairs, such that
$a(\Delta)=a$ and one pair form a conjugacy class, and
$\mathfrak{L}_2(a)$ be the family of pairs, such that $a(\Delta)=a$
and two pairs form a conjugacy class. By $\mathfrak{L}_1$ and
$\mathfrak{L}_2$ denote the union of $\mathfrak{L}_1(a)$ and
$\mathfrak{L}_2(a)$ over all values of $a$ respectively. Certainly,
\begin{equation}\label{eqG4-1}
c_{\pi_1(\mathcal{G}_2),\pi_1(\mathcal{G}_4)}(n)=|\mathfrak{L}_1|+\frac{|\mathfrak{L}_2|}{2}=\frac{|\mathfrak{L}_1|}{2}+\frac{|\mathfrak{L}_1|+|\mathfrak{L}_2|}{2}.
\end{equation}

\begin{lemma}\label{lemG4-2-1}
$$
|\mathfrak{L}_1|+|\mathfrak{L}_2|=\sigma_2(\frac{n}{2})+2\sigma_2(\frac{n}{4})-3\sigma_2(\frac{n}{4}).
$$
\end{lemma}

\begin{proof}
Consider integer $a$, such that $a\equiv 2 \mod 4$ and $a\mid n$.
\Cref{number of halfes} and \Cref{number of halfes-remark} yields
$|\mathfrak{L}_1(a)|+|\mathfrak{L}_2(a)|=\sigma_1(\frac{n}{a})+3\sigma_1(\frac{n}{2a})$.
If $a$ does not satisfy the above condition, then
$|\mathfrak{L}_1(a)|=|\mathfrak{L}_2(a)|=0$.

Summing over all values of $a$ one gets
$|\mathfrak{L}_1|+|\mathfrak{L}_2|=\sum_{2a\mid n,4\nmid
2a}\Big(\sigma_1(\frac{n}{2a})+3\sigma_1(\frac{n}{4a})\Big)=\sum_{2a\mid
n}\Big(\sigma_1(\frac{n}{2a})+3\sigma_1(\frac{n}{4a})\Big)-\sum_{4a\mid
n}\Big(\sigma_1(\frac{n}{4a})+3\sigma_1(\frac{n}{8a})\Big)=\sigma_2(\frac{n}{2})+2\sigma_2(\frac{n}{4})-3\sigma_2(\frac{n}{4})$.
\end{proof}

\begin{lemma}\label{lemG4-2-2}
$$
|\mathfrak{L}_1|=\sum_{a\mid \frac{n}{2}}
\tau(\frac{n}{2a})-\sum_{a\mid \frac{n}{8}} \tau(\frac{n}{8a}).
$$
\end{lemma}

\begin{proof}
We claim $|\mathfrak{L}_1(a)|=\tau(\frac{n}{a}) +
\tau(\frac{n}{2a})$ if $a\equiv 2 \mod 4$, $|\mathfrak{L}_1(a)|=0$
otherwise. The proof is similar to \Cref{lemG4-5}. Summing over all
values of $a$ we get $|\mathfrak{L}_1|=\sum_{2a \mid n, 4\nmid
2a}\Big(\tau(\frac{n}{2a}) + \tau(\frac{n}{4a})\Big)=\sum_{2a \mid
n}\Big(\tau(\frac{n}{2a}) + \tau(\frac{n}{4a})\Big)-\sum_{4a \mid
n}\Big(\tau(\frac{n}{4a}) + \tau(\frac{n}{8a})\Big)=\sum_{a\mid
\frac{n}{2}} \tau(\frac{n}{2a})-\sum_{a\mid \frac{n}{8}}
\tau(\frac{n}{8a})$.
\end{proof}

Substituting \Cref{lemG4-2-1} and \Cref{lemG4-2-2} to
(\ref{eqG4-1}), obtain
$$
\begin{aligned} &
c_{\pi_1(\mathcal{G}_2),\pi_1(\mathcal{G}_4)}(n)=\frac{1}{2}\Big(\sigma_2(\frac{n}{2})+2\sigma_2(\frac{n}{4})-3\sigma_2(\frac{n}{4})+\sum_{a\mid
\frac{n}{2}} \tau(\frac{n}{2a})-\sum_{a\mid \frac{n}{8}}
\tau(\frac{n}{8a})\Big)=\\&\frac{1}{2}\Big(\sigma_2(\frac{n}{2})+2\sigma_2(\frac{n}{4})-3\sigma_2(\frac{n}{4})+\sum_{a\mid
\frac{n}{2}} \tau(a)-\sum_{a\mid \frac{n}{8}} \tau(a)\Big).
\end{aligned}
$$

\subsubsection{Case (i)}
Let $\Delta$ be a subgroup of index $n$ in $\pi_1(\mathcal{G}_4)$
isomorphic to $\pi_1(\mathcal{G}_4)$. The proof is analogous to case
(ii).

For an odd $a\mid n$ the number of conjugacy classes of subgroups
$\Delta$ such that $a(\Delta)=a$ is equal to $\tau(\frac{n}{a})$ if
$\frac{n}{a}$ is odd and is equal to $2\tau(\frac{n}{a})$ if
$\frac{n}{a}$ is even. Since $\tau(\frac{n}{2a})=\tau(\frac{n}{a})$
if $\frac{n}{a}$ is even and $\tau(\frac{n}{2a})=0$ if $\frac{n}{a}$
is odd, we get
$$
\aligned &
c_{\pi_1(\mathcal{G}_4),\pi_1(\mathcal{G}_4)}(n)=\sum_{a\mid n,
2\nmid a}\Big(\tau(\frac{n}{a})+\tau(\frac{n}{2a})\Big)=\sum_{a\mid
n}\Big(\tau(\frac{n}{a})+\tau(\frac{n}{2a})\Big) - \sum_{2a\mid
n}\Big(\tau(\frac{n}{2a})+\tau(\frac{n}{4a})\Big)=\\&\sum_{a\mid
n}\tau(\frac{n}{a}) - \sum_{a\mid
\frac{n}{4}}\tau(\frac{n}{4a})=\sum_{a\mid n}\tau(a) - \sum_{a\mid
\frac{n}{4}}\tau(a).
\endaligned
$$

\section*{Acknowledgment}
The authors are grateful to V. A. Liskovets, R. Nedela, M.~Shmatkov
for helpful discussions during the work of this paper. We also thank
the unknown referee for his/her valuable comments and suggestions.

\section*{Appendix 1}
The numerical values of functions  given by Theorems 1--4 for small
argument are represented in the following table.\smallskip

\noindent \begin{tabular}{|c|clllllllllllllll|}
\hline \textbf{n }&1&2&3&4&5&6&7&8&9&10&11&12&13&14&15&16\\
\cline{1-17}
\textbf{$s_{\pi_1(\mathcal{G}_2),\pi_1(\mathcal{G}_2)}(n)$} &1 &6 &13 &28 &31 &78 & 57 & 120 & 130 & 186 & 133 & 364 & 183& 342& 403& 496\\
\cline{1-17}
\textbf{$c_{\pi_1(\mathcal{G}_2),\pi_1(\mathcal{G}_2)}(n)$} &1 &6 &5 &16 &7 &30 & 9 & 36 & 18 & 42 & 13 & 80 & 15& 54& 35& 76\\
\cline{1-17}
\textbf{$s_{\ZZ^3,\pi_1(\mathcal{G}_2)}(n)$} & &1 & &7 & &13 &  & 35 &  & 31 &  & 91 & & 57& & 155\\
\cline{1-17}
\textbf{$c_{\ZZ^3,\pi_1(\mathcal{G}_2)}(n)$} & &1 & &7 & &9 &  & 29 &  & 19 &  & 63 & & 33& & 107\\
\cline{1-17}
\textbf{$s_{\pi_1(\mathcal{G}_4),\pi_1(\mathcal{G}_4)}(n)$} &1 &2 &1 &4 &11 &2 & 1 & 8 & 10 & 22 & 1 & 4 & 27& 2& 11& 16\\
\cline{1-17}
\textbf{$c_{\pi_1(\mathcal{G}_4),\pi_1(\mathcal{G}_4)}(n)$} &1 &2 &1 &2 &3 &2 & 1 & 2 & 2 & 6 & 1 & 2 & 3& 2& 3& 2\\
\cline{1-17}
\textbf{$s_{\pi_1(\mathcal{G}_2),\pi_1(\mathcal{G}_4)}(n)$} & &1 & &6 & & 13 &  & 28 &  & 31 &  & 78 &  & 57 & &120\\
\cline{1-17}
\textbf{$c_{\pi_1(\mathcal{G}_2),\pi_1(\mathcal{G}_4)}(n)$} & &1 & &4 &  & 3 &  & 9 &   & 5   &  & 16 &  & 5&  & 19\\
\cline{1-17}
\textbf{$s_{\ZZ^3,\pi_1(\mathcal{G}_4)}(n)$} & & & &1 & & &  & 7 &  &  &  & 13 & & & & 35\\
\cline{1-17}
\textbf{$c_{\ZZ^3,\pi_1(\mathcal{G}_4)}(n)$} & & & &1 & & &  & 5 &  &  &  & 5 & & & & 17\\
\hline \cline{1-17} \hline \cline{1-17}
\end{tabular}
\begin{center}

\small{Table~1}
\end{center}

\bigskip

We note some properties of the above functions. The proofs follows
by direct calculation based on the explicit formulas.

A function $f(n)$ is called {\em multiplicative} if $f(kl)=f(k)f(l)$
for coprime integers $k$, $l$. The functions
$s_{\pi_1(\mathcal{G}_2),\pi_1(\mathcal{G}_2)}(n)$,
$c_{\pi_1(\mathcal{G}_2),\pi_1(\mathcal{G}_2)}(n)$,
$s_{\pi_1(\mathcal{G}_4),\pi_1(\mathcal{G}_4)}(n)$ and
$c_{\pi_1(\mathcal{G}_4),\pi_1(\mathcal{G}_4)}(n)$ are
multiplicative. The following functions are also multiplicative $n
\to s_{\ZZ^3,\pi_1(\mathcal{G}_2)}(2n)$ and $n \to
s_{\ZZ^3,\pi_1(\mathcal{G}_4)}(4n)$.
%
%
%For some other mentioned functions close relations holds: the
%functions $n \to s_{\ZZ^3,\pi_1(\mathcal{G}_2)}(2n)$ and $n \to
%s_{\ZZ^3,\pi_1(\mathcal{G}_4)}(4n)$ are multiplicative.
%$$
%s_{\ZZ^3,\pi_1(\mathcal{G}_2)}(2kl)=s_{\ZZ^3,\pi_1(\mathcal{G}_2)}(2k)s_{\ZZ^3,\pi_1(\mathcal{G}_2)}(2l),
%$$
%$$
%c_{\ZZ^3,\pi_1(\mathcal{G}_2)}(2kl)=c_{\ZZ^3,\pi_1(\mathcal{G}_2)}(2k)c_{\ZZ^3,\pi_1(\mathcal{G}_2)}(2l),
%$$
%$$
%s_{\ZZ^3,\pi_1(\mathcal{G}_4)}(4kl)=s_{\ZZ^3,\pi_1(\mathcal{G}_4)}(4k)s_{\ZZ^3,\pi_1(\mathcal{G}_4)}(4l),
%$$
%$$
%c_{\ZZ^3,\pi_1(\mathcal{G}_4)}(4kl)=c_{\ZZ^3,\pi_1(\mathcal{G}_4)}(4k)c_{\ZZ^3,\pi_1(\mathcal{G}_4)}(4l).
%$$

%For $s_{\ZZ^3,\pi_1(\mathcal{G}_2)}(n)$,
%$c_{\ZZ^3,\pi_1(\mathcal{G}_2)}(n)$,
%$s_{\ZZ^3,\pi_1(\mathcal{G}_4)}(n)$ and
%$c_{\ZZ^3,\pi_1(\mathcal{G}_4)}(n)$ hold the close relations:
%$s_{\ZZ^3,\pi_1(\mathcal{G}_2)}(2km)=s_{\ZZ^3,\pi_1(\mathcal{G}_2)}(2k)s_{\ZZ^3,\pi_1(\mathcal{G}_2)}(2m)$,
%$c_{\ZZ^3,\pi_1(\mathcal{G}_2)}(2km)=c_{\ZZ^3,\pi_1(\mathcal{G}_2)}(2k)c_{\ZZ^3,\pi_1(\mathcal{G}_2)}(2m)$,
%$s_{\ZZ^3,\pi_1(\mathcal{G}_4)}(4km)=s_{\ZZ^3,\pi_1(\mathcal{G}_4)}(4k)s_{\ZZ^3,\pi_1(\mathcal{G}_4)}(4m)$,
%$c_{\ZZ^3,\pi_1(\mathcal{G}_4)}(4km)=c_{\ZZ^3,\pi_1(\mathcal{G}_4)}(4k)c_{\ZZ^3,\pi_1(\mathcal{G}_4)}(4m)$.

\section*{Appendix 2}
As it was mentioned to us by the referee, one of convenient tools to
deal with number-theoretical functions arising in crystallography is
the Dirichlet generating function. See for example \cite{Baake}.
Given a sequence $\{f(n)\}_{n=1}^\infty$, a formal power series
$$
\widehat{f}(s)=\sum_{n=1}^\infty\frac{f(n)}{n^s}
$$
is called the Dirichlet generating function for
$\{f(n)\}_{n=1}^\infty$, see (\cite{Apostol}, Ch. 12). To
reconstruct the sequence $f(n)$ from $\widehat{f}(s)$ one can use
Perron's formula (\cite{Apostol}, Th. 11.17).

Here we present the Dirichlet generating functions for the sequences
$s_{H,G}(n)$ and $c_{H,G}(n)$. Since theorems 1--4 provide the
explicit formulas, the remaining is done by direct calculations.

By $\zeta(s)$ we denote the Riemann zeta function
$\displaystyle\zeta(s)=\sum_{n=1}^{\infty}\frac{1}{n^s}$. Following
\cite{Apostol} note
$$\widehat{\sigma}_0(s) = \zeta^2(s),\quad\widehat{\sigma}_1(s) =\zeta(s)\zeta(s-1),\quad \widehat{\sigma}_2(s) = \zeta^2(s)\zeta(s-1),\quad \widehat{\omega}(s) = \zeta(s)\zeta(s-1)\zeta(s-2).$$

Define sequence $\{\chi(n)\}_{n=1}^\infty$ by
 $\chi(n)=\left\{\begin{aligned}
1  \;\text{if}\; n\equiv 1 \mod 4\\
0  \quad\quad\quad\quad\;\text{if}\; 2\mid n\\
-1  \;\text{if}\; n\equiv 4 \mod 4\\
\end{aligned}
\right. $, and put $\eta(s)=\widehat{\chi}(s)$. Note that $\eta(s)$
is the Dirichlet L-series for the multiplicative character
$\chi(n)$. Then $\widehat{\tau}(s)=\zeta(s)\eta(s)$. In more
algebraic terms,
$$
\widehat{\tau}(s)=\frac{1}{1-2^{-s}}\prod_{p\equiv 1 \mod
4}\frac{1}{(1-p^{-s})^2}\prod_{p\equiv 3 \mod
4}\frac{1}{1-p^{-2s}}\,.
$$
For details see (\cite{Baake} p.~4).

In the following table we provide the Dirichlet generating functions
for the sequences given by Theorems 1--4.
\def\formi{$2^{-s}\zeta(s)\zeta(s-1)\zeta(s-2)$}
\def\formii{$4^{-s}\zeta(s)\zeta(s-1)\zeta(s-2)$}
\def\formiii{$2^{-s-1}\zeta(s)\zeta(s-1)\Big(\zeta(s-2)+
 {(1+3\cdot2^{-s})}\zeta(s)\Big)$}
\def\formiv{$\quad\quad\quad\quad\quad\quad 2^{-2s-2}\zeta(s)\Big(\zeta(s-1)\zeta(s-2)+{(1+3\cdot2^{-s})\zeta(s)\zeta(s-1)+2(1+2^{-s})\zeta(s)\eta(s)}\Big)$}
\def\formv{$\big(1-2^{-s}\big)\zeta(s)\zeta(s-1)\zeta(s-2)$}
\def\formva{$2^{-s}\big(1-2^{-s}\big)\zeta(s)\zeta(s-1)\zeta(s-2)$}
\def\formvi{$\big(1-2^{-s}\big)\big(1+3\cdot2^{-s}\big)\zeta(s)^2\zeta(s-1)$}
\def\formvia{$2^{-s-1}(1-2^{-s})\zeta(s)^2\Big((1+3\cdot2^{-s})\zeta(s-1)+{(1+2^{-s})\eta(s)}\Big)$}
\def\formvii{$\big(1-3^{-s}\big)\zeta(s-1)^2\vartheta(s-1)$}
\def\formviii{$2^{-s}\big(1-3^{-s}\big)\zeta(s-1)^2\vartheta(s-1)$}
\def\formix{$\big(1-3^{-s}\big)\big(1+2\cdot3^{-s}\big)\zeta(s)^2\vartheta(s)$}
\def\formx{$2^{-s}\big(1-3^{-s}\big)\big(1+3^{-s}\big)\zeta(s)^2\vartheta(s)$}
\def\formxi{$\big(1-2^{-s}\big)\zeta(s)\zeta(s-1)\eta(s-1)$}
\def\formxii{$\big(1-2^{-s}\big)\big(1+2^{-s}\big)\zeta(s)^2\eta(s)$}
$$
\begin{array}{|c|c|p{5.8cm}|p{8.2cm}|}\hline
 \multicolumn{2}{|c|}{$\backslashbox{H}{G}$} & $\mathcal{G}_{2}$ & $\mathcal{G}_{4}$ \\ \hline
 \multirow{2}*{$\ZZ^3$} & \widehat{s}_{H,G} & \formi & \formii \\ \cline{2-4}
 & \widehat{c}_{H,G} & \formiii & \formiv  \\ \hline
 \multirow{2}*{$\mathcal{G}_{2}$} & \widehat{s}_{H,G} &\formv  & \formva \\ \cline{2-4}
 & \widehat{c}_{H,G} &\formvi  & \formvia  \\ \hline
 \multirow{2}*{$\mathcal{G}_{4}$} & \widehat{s}_{H,G} & \quad do not exist & \formxi \\ \cline{2-4}
 & \widehat{c}_{H,G} & \quad do not exist & \formxii  \\ \hline
\end{array}
$$
\begin{center}

\small{Table~2}
\end{center}

\end{document}